\documentclass[12pt]{article}
\usepackage{graphicx}
\usepackage{amsmath}
\usepackage{amsfonts}
\usepackage{amsthm}
\usepackage{bm}
\usepackage{color}
\usepackage{verbatim} 
\usepackage[hidelinks]{hyperref} 
\usepackage{url}

\usepackage{numprint}
\npfourdigitnosep
\npthousandsep{\hspace*{1pt}} 
\npdecimalsign{.}

\newcommand{\A}[1]{\href{http://oeis.org/A#1}{A#1}}
\newtheorem{theorem}{Theorem}[section]
\newtheorem{lemma}[theorem]{Lemma}

\newtheorem{proposition}[theorem]{Proposition}

\theoremstyle{definition}
\newtheorem{definition}[theorem]{Definition}

\newcommand{\inclrange}[2]{[#1, $\dots$, #2]}

\begin{document}

\title{Archive Labeling Sequences}

\author{Tanya Khovanova, Gregory Marton}
\date{}
\maketitle

\begin{abstract}
What follows is the story of a family of integer sequences, which started life as a Google interview puzzle back in the previous century when VHS video tapes were in use.
\end{abstract}


\section{Google's Puzzle}

\begin{quote}
Suppose you are buying VHS tapes and want to label them using the stickers that came in the package. You want to number the tapes consecutively starting from 1, and the stickers that come with each package are exactly one of each digit \inclrange{\boxed{0}}{\boxed{9}}. For your first tape, you use only the digit \boxed{1} and save all the other digit stickers for later tapes. The next time you will need a digit \boxed{1} will be for tape number 10. By this time, you will have several unused \boxed{1} stickers. What is the next tape number such that after labeling the tape with that number, you will not have any \boxed{1} stickers remaining?
\end{quote}

\section{Ones Counting Function}

The puzzle appeared in Google Labs Aptitude Test \cite{GLAT} in the following formulation.

\begin{quote}
Consider a function $f$ which, for a given whole number $x$, returns the number of ones required when writing out all numbers between 0 and $x$ inclusive. For example, $f(13) = 6$. Notice that $f(1) = 1$. What is the next largest $x$ such that $f(x) = x$? 
\end{quote}

One might notice that it is unclear that the two questions above are equivalent since the happy owner of the tapes might hypothetically run out of another sticker, say sticker \boxed{2} first, thus not reaching a point where all stickers \boxed{1} are used. Intuitively, we can expect that sticker \boxed{1} is the first to run out, and we will prove this shortly.

It is also unclear if any $x$ even exists such that $f(x) = x$. In teaching, we often ask students to find things that do not exist, expecting a proof of non-existence. While such problems may be considered evil, they are legitimate. At the time, Google's unofficial motto was ``Don't be Evil'', and they weren't: we will see that the answer does indeed exist.

But we digress. Our function $f(x)$ is the number of \boxed{1} stickers needed to label all the tapes up to tape $x$. When $f(x) = x$, then we have used all of the \boxed{1} stickers in labeling the first $x$ tapes. The function $f(x)$ can be found in the Online Encyclopedia of Integer Sequences \cite{OEIS} as sequence \A{094798}.

In considering the other non-zero digits, let $f_d(x)$ count the number of \boxed{d} stickers needed to number the first $x$ tapes (and of course let $f(x)$ henceforth be $f_1(x)$). In the single and double-digit numbers, there are ten of each non-zero digit in the ones column and ten in the tens column, so 20 altogether. Early on, the tape number is ahead of the digit count. By the time we get to 20-digit numbers, though, there should be, on average, two of any single non-zero digit per number.\footnote{We're looking at non-zero digits for now only because one would not use stickers for leading zeroes, unlike other leading digits, but we will return to zeroes shortly.}  Thus, the number of times that any digit is used should eventually catch up with the tape numbers.

Encouraged by assurance of reaching our goal somewhere, we might continue our estimate. In the up-to three-digit numbers, those less than $10^4$, there are 300 of each non-zero digit; in the numbers below $10^5$, there are \numprint{4000}; then \numprint{50000} below $10^6$, and so on up to $10^{10}$, where $f_{d>1}(x)$ and $x$ must (almost) meet. In particular, there are \numprint{10000000000} counts for any non-zero digit in the numbers below \numprint{10000000000}. Hence, were the puzzle asking about any of the digits 2--9, then ten billion could have been an easy answer or, at least a limit on how far we need to search.

Sadly, there is a 1 in the decimal representation of ten billion (and a few zeroes), so we require $10^{10}+1$ digits \boxed{1} to write the numbers $\inclrange{1}{10^{10}}$. Thus, $f_1(10^{10}) \ne 10^{10}$, so $10^{10}$ cannot be the answer to the original puzzle. Thus stymied, we wrote a program to find the solution to the original Google puzzle. And the answer turned out to be 199981, much smaller than we expected.

\section{Counting Other Digits}

We were so enjoying our stymie\footnote{Yes, we just nouned that verb.} that we then wrote a program to solve the puzzle for any non-zero digit.

\begin{definition}
We denote by $a_=(d)$ the smallest number $x > 1$ such that the decimal representation of \boxed{d} appears as a substring of the decimal representations of the numbers $\inclrange{1}{x}$ exactly $x$ times:
\[a_{\boldsymbol{=}}(d) = \min(\{x > 1: f_d(x) \boldsymbol{=} x\}).\]
\end{definition}

We already know that $a_=(1)$ is \numprint{199981}. The sequence $a_=(d)$, which now has number \A{163500}, continues as follows:
\begin{multline*}
\numprint{28263827},\ \numprint{371599983},\ \numprint{499999984},\ \numprint{10000000000},\\
\numprint{9500000000},\ \numprint{9465000000},\ \numprint{9465000000},\ \numprint{10000000000}.
 \end{multline*}

Did you expect this sequence to be increasing? You could have because smaller numbers tend to contain smaller digits than larger numbers. Then why is the sequence not increasing? As we failed to find a value for the digit \boxed{5} below ten billion, we noticed that it is fairly easy to imagine a scenario where you have one less than the number you need, and then the next value has more than you need for equality, and then you equalize again later. In response, we decided to look at a related sequence.
\begin{definition}
Let \[a_{\boldsymbol{>}}(d) = \min(\{x : f_d(x) \boldsymbol{>} x\}).\]
\end{definition}

 The key difference is in using ``more than'' rather than ``exactly''. Thus, we will also call our $a_=(d)$ sequence the ``exactly'' sequence and our $a_>(d)$ the ``more than'' sequence. 

We later discovered that this related sequence was published at IBM's famous puzzle website \href{https://research.ibm.com/haifa/ponderthis/challenges/April2004.html}{``Ponder This'' in April 2004} and was authored by Michael Brand \cite{PT}. This version is quite natural as it wonders when we first run out of the labels. Moreover, the \boxed{1} sticker plays a special role in this puzzle as it must be the digit that will run out first, as we see in the following table and shall prove theoretically in Proposition~\ref{prop:nondecreasing}.

Starting at 1, Table~\ref{table:2seq} shows the first nine terms of the ``exactly'' and ``more than'' sequences.

\begin{table}[ht!]
\begin{center}
\begin{tabular}{ c r r }
  $d$ & $a_=(d)$ & $a_>(d)$ \\
  \boxed{1} & \numprint{199981} & \numprint{199991} \\
  \boxed{2} & \numprint{28263827} & \numprint{28263828} \\
  \boxed{3} & \numprint{371599983} & \numprint{371599993} \\
  \boxed{4} & \numprint{499999984} & \numprint{499999994} \\
  \boxed{5} & \numprint{10000000000} & \numprint{5555555555} \\
  \boxed{6} & \numprint{9500000000} & \numprint{6666666666} \\
  \boxed{7} & \numprint{9465000000} & \numprint{7777777777} \\
  \boxed{8} & \numprint{9465000000} & \numprint{8888888888} \\
  \boxed{9} & \numprint{10000000000} & \numprint{9999999999} \\
\end{tabular}
\end{center}
\caption{The first nine terms of $a_=(d)$ and $a_>(d)$.}
\label{table:2seq}
\end{table}

Some of these rows are interesting in their own right. Notice that \numprint{199991} is ten more than the previously found \numprint{199981}. For all the numbers in between, the initial equality holds ($\forall i\in\inclrange{199981}{199991}$ we have $i = f_1(i)$). Likewise, for $d=\boxed{3}$, each of the numbers between \numprint{371599983} and \numprint{371599993} has exactly one three, so the increase in a number by one is the same as the increase in the count of threes. A similar situation holds for $d=\boxed{4}$.

The sequence $a_>$ can be found using the identifier \A{164321} in the OEIS. Unsurprisingly, the values matching this relaxed second condition are more well-behaved than those with equality.

Did you notice that the second column is increasing? This might be surprising for the fans of the Champernowne constant. What's the Champernowne constant? Imagine you placed an infinitude of labeled VHS tapes in order. The labels together will read as a concatenation of all positive integers, whose digits form the sequence A033307. Now we add a zero with a dot in front to get the constant:
\[0.12345678910111213141516\ldots.\]
The constant is most famous for being a ``normal'' number in any base \cite{BC}. Here \textit{normal} is a mathematical term referring to the distribution of digits. Normal means that all possible strings of digits of the same length have the same density. This means that every digit in base 10 appears with the same density. Despite this, our second column is increasing, demonstrating an unsurprising fact that smaller digits appear earlier than the bigger digits.  

\section{More ``Exactly'' Sequences}

We want to introduce a few more related sequences, one per digit, where the letter E symbolizes exactness or equality.

\begin{definition}
    Let $E_d$ be an increasing sequence of positive integers $x$ such that $f_d(x) = x$.
\end{definition}
The sequence $E_d$ must be finite. After all, starting from 11-digit numbers, the supply of labels starts decreasing. We have to run out of labels. We can be more precise in claiming that the largest value in $E_d$ is not more than $d10^{10}$, which we prove in a more general setting in Proposition~\ref{prop:bounds}.

Sequences $E_d$ are connected to our sequence $a_=(d)$:
\begin{equation*}
a_=(d) =
\begin{cases}
E_d(2) & \text{for } d = \boxed{1} \\
E_d(1) & \text{otherwise}
\end{cases}.
\end{equation*} 

Recall, the special case for \boxed{1} is what made the puzzle interesting because $E_1(1)=1$. The sequences $E_d$ are in the OEIS database, and we show their A-numbers and lengths in Table~\ref{table:seqnum}.

\begin{table}[ht!]
\begin{center}
\begin{tabular}{ c c c }
  $d$ & OEIS ref. for $E_d$ & Number of terms \\
  \boxed{1} & \A{014778} & 83 \\
  \boxed{2} & \A{101639} & 13 \\
  \boxed{3} & \A{101640} & 35 \\
  \boxed{4} & \A{101641} & 47 \\
  \boxed{5} & \A{130427} & 4 \\
  \boxed{6} & \A{130428} & 71\\
  \boxed{7} & \A{130429} & 48 \\
  \boxed{8} & \A{130430} & 343 \\
  \boxed{9} & \A{130431} & 8 \\
\end{tabular}
\end{center}
\caption{The sequence numbers for $E_d$ and their lengths.}
\label{table:seqnum}
\end{table}

The numbers of terms are their own sequence! It appears in the OEIS in disguise: sequence \A{130432} is the last column of Table~\ref{table:seqnum} plus 1, because the sequence author assumed that tapes would be numbered starting with 0. While that choice may have tempted the audience of this paper\footnote{If you numbered your VHS tapes starting at zero, please send a note, kindred spirit!}, it would not have been common practice. However, if we did start at zero, and thus add 1 to the last column, we see a neat pattern: the result is divisible by $d$. This hides an even more interesting fact: the actual values of $E_d$ are periodic modulo $10^{10}$, while being bounded by  $d\cdot 10^{10}$; the latter fact is proven in Proposition~\ref{prop:bounds}.

To explain periodicity, we observe that for $0 \le x < (d-1)10^{10}$, we have $f_d(x+10^{10}) = f_d(x) + 10^{10}$. It follows that the numbers $x$ and $x + 10^{10}$, are either both members of the sequence $E(d)$ or both non-members. Thus the number of the solutions to the equation $f_d(x) = x$ in the range $\inclrange{0}{10^{10}-1}$ is the same as in the range $\inclrange{r10^{10}}{(r+1)10^{10}}-1$, when $r < d$. Hence, we have $d$ ranges with the same number of solutions, which explains the divisibility of \A{130432}$(d)$ by $d$.

When studying Table~\ref{table:2seq}, you might notice that stickers \boxed{5} and \boxed{9} delay the start of the corresponding exact sequences until the latest possible value of $x$ of \numprint{10000000000}. Not surprisingly, in Table~\ref{table:seqnum} the count for the number of terms for values 5 and 9 is much smaller than for other stickers. Due to the argument in the previous paragraph, all solutions of $f_d(x) = x$ for $d$ equaling 5 or 9 have to be of the form $r10^{10}$, where $r < d$. Thus, the last column of~\ref{table:seqnum} has to be the smallest possible value of exactly $d-1$.

Now that the upper bound is clear, we can find the largest values and treat them as another sequence, shown in Table~\ref{table:E_d}.

\begin{table}[ht!]
\begin{center}
\begin{tabular}{r r}
   $d$ & $\max(E(d))$ \\
    \boxed{1} & \numprint{1111111110},\\ 
    \boxed{2} & \numprint{10535000000},\\
    \boxed{3} & \numprint{20500000000},\\
    \boxed{4} & \numprint{30500000000},\\
    \boxed{5} & \numprint{40000000000},\\
    \boxed{6} & \numprint{59628399995},\\
    \boxed{7} & \numprint{69971736170},\\
    \boxed{8} & \numprint{79998399997},\\ 
    \boxed{9} & \numprint{80000000000}.
\end{tabular}
\caption{Largest values of $x$, where $f_d(x) = x$.}
\label{table:E_d}
\end{center}
\end{table}

Let's now dive deeper into the $d=0$ case.

\section{Counting Zeroes}
\label{sec:zeros}

In counting zeroes, let us recall that the puzzle specifies that the first VHS tape is labeled with the \boxed{1} sticker, not \boxed{0}. Expanding on $f$, we denote the function that calculates zeroes in numbers 1 through $x$ inclusive as $f_0(x)$. It is represented in the OEIS as sequence \A{061217}.

We calculated that the smallest number $x$ such that $x$ is less than or
equal to the number of 0s in the decimal representations of $\inclrange{1}{x}$ is $\numprint{100559404366}$, equivalently this number is $a_>(0)$. But what is the corresponding number for the $a_=$ sequence? It appears that no such number exists. To prove it, we need to start with a lemma.

\begin{lemma}
\label{lemma:step}
For any integer $x > 10^{10}$, we have $f_0(x+10^{10}) \ge f_0(x) + 10^{10}$.
\end{lemma}

\begin{proof}
Indeed, numbers between $x$ and $x+10^{10}$ go through all possible combinations of the last ten digits. Hence, they contain at least $10^{10}$ zeroes.
\end{proof}

Now we are ready to prove our theorem.

\begin{theorem}
The value $a_=(0)$ is not well-defined.
\end{theorem}

\begin{proof}
We calculated that $f_0(\numprint{100559404366}) = \numprint{100559404367}$. Its predecessor then must be $f_0(\numprint{100559404365})=\numprint{100559404364}$ with three fewer zeros. We verified that there were no equalities up to this point, and indeed up to a bigger number, but of course we couldn't continue checking up to infinity. 

So we need other arguments. Notice that number $\numprint{100559404366}$ has three zeroes. Hence, for some $y$ that are not much bigger than $\numprint{100559404367}$, we will have that $f_0(y+1) \ge f_0(y) +3$. For some time, the sequence $f_0$ will be increasing in steps not less than three. We are getting away from the equality at high speed. 

Were we dealing with random 12-digit numbers, then such numbers would have on average $11/10$ zeroes. Hence, $f_0(x)$ grows faster than $x$ at this point. But this consideration is not a proof. To finish the proof of the theorem, we need to find a number $y > 10^{10}$ such that $f_0(y) > y + 10^{10}$ and check that there is no solution to $f_0(x) = x$ below $y$. By Lemma~\ref{lemma:step}, that number $y$ would guarantee that $f_0(x)$ will always be ahead of its index after $y$. 

Let us find such a number. We start with \numprint{100559404366}. The sequence $f_0(x)$ will continue to grow not slower than its index $x$  until the next number that doesn't contain zeroes. Such a number is \numprint{111111111111}. We calculated that $f_0(\numprint{111111111111})= \numprint{120987654321}$. So the number of zeroes is way ahead of the number itself. As the sequence $f_0(x)$ is non-decreasing, we can't have $y$ such that $f_0(y) = y$ until \numprint{120987654321}. This way, we can speed up the process, and we need a small number of iterations to get to such a number. We performed appropriate calculations, thus concluding the proof of the theorem.
\end{proof}

\section{Greater or Equal}

In addition to $a_=$ and $a_>$, we counted the ``greater or equal'' sequence $a_\ge(d)$, where $d$ again denotes the sticker in question. The great property of this latter sequence is that 
\[a_\ge(d) = \min(a_=(d), a_>(d)).\]
This sequence appears in the database as sequence \A{164935}. How can we define such a sequence for multi-digit stickers? The idea is to ignore stickers and consider multi-digit strings; for which we give proper definitions in a later section.

One more caveat: we defined $a_=(1)$ to be the smallest number greater than 1 satisfying the VHS property. This complicated condition was needed so that the sequence would include the solution of Google's puzzle, \numprint{199981}, as the first term. But \A{164935}$(1) = 1$ as it should be. This sequence is non-decreasing for the same reason the ``more than'' sequence is non-decreasing. We prove this in Proposition~\ref{prop:nondecreasing}.

\section{The Algorithms}

So that you may easily check the facts we have described, we would like to share the \href{https://colab.research.google.com/drive/1pGfgQWvJR1IAG3t4dNnrTnc07UvyV4xC}{algorithms we used} \cite{MartonKhovanova2023}. In this section, we describe a more efficient way to find $f_d(x)$. We counted the digit $d$ separately in each decimal place it occurred. Suppose we want to count how many times the digit $d$ occurred in the $k$-th place from the right in the set $\inclrange{1}{x}$. It depends on which digit the number $x$ has in the $k$-th place from the right. Suppose this digit is $x_k$. Consider the number $y = \lfloor x/10^k\rfloor 10^{k}$. We chose $y$ because it is the largest number not exceeding $x$ with $k$ zeros at the end. In the range $\inclrange{1}{y-1}$, if we pad smaller integers with zeros on the left, each digit appears in the $k$-th place from the right the same number of times. Therefore, any digit $d>0$ appears in this range $\frac{y}{10} = \lfloor x/10^k\rfloor 10^{k-1}$ times.

Now, we need to calculate how often $d$  appears in the place of interest in the range $\inclrange{y}{x}$. If $x_k < d$, then it doesn't appear at all. If $x_k > d > 0$ we need to add $10^{k-1}$. If $x_k = d > 0$, we need to add the total count of our digit in the range, which is $(x \mod 10^{k-1}) + 1$.

We need to consider the case of $d=0$ separately, as we should not count leading zeros, nor zero itself, as the sequence starts at 1. If $x_k > d = 0$, the count is $\lfloor x/10^k\rfloor 10^{k-1}$, (the same as the $x_k<d$ case for other digits), but if the $k$-th digit is zero, we need to subtract the number of digits in the range $\inclrange{1}{y-1}$ that have fewer than $k$ digits and add the number of digits in the range $\inclrange{y}{x}$ that have 0 in the $k$-th place from the right. Thus the adjusment is $ - 10^{k-1} + (x \mod 10^{k-1}) + 1$.

To summarize, we would like to express $f_d(x)$ as the sum of the contributions $c_d(x_k)$ of the counts of the digit $d$ in the $k$-the place from the right. This contribution depends on the value of $x_k$. Let $Y$ be shorthand for $\lfloor x/10^k\rfloor \cdot 10^{k-1}$, then:
\newcommand{\whenand}[2]{\text{when } #1 \text{ and } #2}
\begin{equation*}
c_d(x_k) = 
\begin{cases}
Y                                  & \whenand{d > 0}{x_k < d} \\
Y + (x \mod 10^{k-1}) + 1           & \whenand{d > 0}{x_k = d} \\
Y + 10^{k-1}                        & \whenand{d > 0}{x_k > d} \\
Y                                  & \whenand{d = 0}{x_k > d} \\
Y - 10^{k-1} + (x \mod 10^{k-1}) + 1 & \whenand{d = 0}{x_k = d}
\end{cases}.
\end{equation*}
Summing over each $k$-th place, we get
\begin{equation}
f_d(x) = \sum_k{c_d(x_k)}.
\end{equation}

We can now use this closed-form for $f_d(x)$ in much faster searches for $a_\ge(d)$. To do so, we need the following lemma that allows us to skip a lot of numbers in our search.

\begin{lemma}
\label{lemma:safe}
Suppose we already know that $a_\ge(d) > x$. Suppose, in addition, we can show that $f_d(y) < x$ for some $y > x$. Then $a_\ge(d) > y$.
\end{lemma}

\begin{proof}
As $f_d$ is non-decreasing and $f_d(y) < x$, we know that the value of function $f_d$ on any element in the range $\inclrange{x}{y}$ is not greater than $x$. It follows that $a_\ge(d) > y$.
\end{proof}

We search the infinite space of possible values using a variation of unbounded binary search \cite{bentley_yao_1976}. We call a range of numbers $\inclrange{x}{x+p}$ ``safeleft'' if we can guarantee that $a_\ge(d) > x$. We start with a safeleft range \inclrange{2}{3}. When $d=1$, we can't start with the range whose left side is 0, as we will get the answer 1, which we want to skip. It is easy to see that the base case holds for 2 in other words, $f_d(2) < 2$ for any $d$, as we only use one \boxed{1} and one \boxed{2} sticker up to tape number 2. Then we iterate to the next safeleft range as follows:
\begin{itemize}
\item If $f_d(x+p) < x$, then $a_{\ge}(d)$ is not in the range by Lemma~\ref{lemma:safe}, making any range starting with $x+p$ safeleft. The next range to search is $\inclrange{x+p}{x+3p}$, where we move the start of the range to $x+p$ and increase the size of the range twice.
\item If $f_d(x+p) \ge x$, then $a_{\ge}(d)$ is not guaranteed to be outside of the range. The next range to search is $\inclrange{x}{x+p/2}$, where we keep the start of the range and halve the size of the range.
\item Suppose we reduced the range size to 1. Then if $f_d(x) < x$ and $f_d(x+1) \geq x+1$, we have $a_\ge(d) = x+1$. If not, then any range starting with $x+1$ is safe, and the new range is $\inclrange{x+1}{x+3}$.
\end{itemize}

After the value of $a_{\ge}(d)$ is found, finding the value of $a_{>}(d)$ is easy for non-zero digits. One may need to check several next values. For zero it is not as easy, but see Section~\ref{sec:zeros}. When looking for the exact sequence $a_=(d)$, the answer is not always near $a_{\ge}(d)$, but we can still search rapidly. If we already showed that $a_=(d) > x$  and if $f_d(x) > x$, then $a_=(d) \ge f_d(x)$. After all, if we saw no digits $d$ in the range $\inclrange{x}{f_d(x)-1}$ at all, $x$ would not catch up to $f_d(x)$ below $f_d(x)$.

\section{Multiple Digits}

There is no reason that we should be constrained to single digits. The formal statement of the problem provides a generalization, where we consider substrings of each of the numbers $\inclrange{1}{x}$ rather than digits in those numbers. We should note that we count every occurrence of a substring separately. Thus \boxed{11} will be counted twice as a substring of \numprint{1113} even though an actual sticker with ``11'' printed on it could be used in either position, but not in both positions simultaneously.

Now that we defined the ``more than'' sequence $a_>$ for any positive integer, we can prove the statement we promised before, along with a corresponding statement about the $a_\ge$ sequence.

\begin{proposition}
\label{prop:nondecreasing}
    The ``more than'' sequence $a_>$ and the ``greater or equal'' sequence $a_\ge$ are non-decreasing after the first terms $a_>(1)$ and $a_\ge(1)$.
\end{proposition}

\begin{proof}
   For two strings $i$ and $j$, if $i<j$, then for every occurrence of $j$ in a number $x$, we can get a smaller number with an occurrence of $i$ by replacing $j$ with $i$. It follows that for $0 < i < j$, and any $x$, 
   \[f_i(x) \ge f_j(x).\]
   It follows that $f_i(a_>(j)) \ge f_j(a_>(j)) = a_>(j)$, and $f_i(a_\ge(j)) \ge f_j(a_\ge(j)) = a_\ge(j)$ implying that $a_>(i) \le a_>(j)$ and $a_\ge(i) \le a_\ge(j)$.
\end{proof}

Inspired, we wrote an even fancier program to find values of the ``more or equal'' sequence $a_{\ge}$ for multi-digit numbers. As before, we start by calculating $f_d(x)$, where $d$ is an $n$-digit number. As a warm-up, we have an exercise for the reader to check that for $k \ge n$ 
\[f_d(10^k - 1) = k10^{k-n}.\]

To calculate $f_d(x)$, we count $d$'s contribution separately in each decimal place it occurred, parametrized by the placement of its last digit. Suppose we want to count how many times the $n$-digit string $d$ occurred so that its last digit is in the $k$-th place from the right in the range $\inclrange{1}{x}$. It depends on which $n$-digits the number $x$ has in the corresponding place. Suppose this $n$-digit number is $x_k$. Consider the number $y = \lfloor x/10^{k+1}\rfloor 10^{k+1}$. In the range $\inclrange{1}{y-1}$, if we pad smaller integers with zeros on the left, each $n$-digit number appears in the $k$-th place from the right the same number of times. Therefore, $d$ appears in this range $\frac{y}{10^n} = \lfloor x/10^{k+n-1}\rfloor 10^{k-1}$ times.

Now, we need to calculate how often $d$  appears in the place of interest in the range $\inclrange{y}{x}$. If $x_k < d$, then it doesn't appear at all. If $x_k > d$ we need to add $10^{k-1}$. If $x_k = d$, we need to add the total count of appearance $d$ in the given spot in the range $\inclrange{y}{x}$, which is $(x \mod 10^{k-1}) + 1$.

To summarize, we would like to express $f_d(x)$ as the sum of the contributions $c_d(x_k)$ of the counts of the $n$-digit sticker $d$ in the $k$-th place from the right. Here we assume that $n > 1$, and, consequently, $d > 0$, which makes the following formula simpler than the one for a single digit. This contribution depends on the value of $x_k$. Let $Y$ be shorthand for $\lfloor x/10^{k-n+1}\rfloor \cdot 10^{k-1}$, then:

\begin{equation*}
c_d(x_k) = 
\begin{cases}
Y                                   & \textrm{ when } x_k < d \\
Y + (x \mod 10^{k-1}) + 1           & \textrm{ when } x_k = d \\
Y + 10^{k-1}                        & \textrm{ when } x_k > d
\end{cases}.
\end{equation*}

After working out the case $d=0$, we were not immediately sure that $a_=(d)$ is defined for every $d > 0$. However, it is.

\begin{theorem}
\label{thm:welldefined}
The value $a_=(d)$ is well-defined for any $d > 0$.
\end{theorem}

\begin{proof}
If $d$ is an $n$-digit sticker that is not a power of 10, then $f_d(10^k - 1) = f_d(10^k)$. From the exercise above, it follows that $f_d(10^k) = k10^{k-n}$. Plugging in $k = 10^n$, we get $f_d(10^{10^n}) = 10^n10^{10^n-n} = 10^{10^n}$. Thus, $10^{10^n}$ is always a solution for $f_d(x)= x$. Therefore, for an $n$-digit sticker $d$ that is not a power of 10, the function $a_=(d)$ is well-defined and $a_=(d) \le 10^{10^n}$.

Now we need to check the case when an $n$-digit sticker $d$ is a power of 10, that is $d = 10^{n-1}$. Consider $x = 2\cdot 10^{10^n-6+n}$, then $Y = 2\cdot 10^{10^n-6}$. We count the contribution of the sticker $d$, where the last digit is in the $k$-th place from the right, where $k$ must be in the range $\inclrange{1}{10^n-4}$.

If $k < 10^n-4$, then $x_k$ is a string of 0s and $d > x_k$, and the contribution for this $k$ is $Y = 2\cdot 10^{10^n-6}$. If $k = 10^n-4$, then $x_k$ is a 2 followed by $n-1$ zeros, so $d < x_k$, and the corresponding $Y$ is 0, and the contribution for this $k$ is $10^{k-1} = 2\cdot 10^{10^n-6} + 10^{10^n-5}$. Summing up, we get 
\[2\cdot 10^{10^n-6}(10^n-5) + 10^{10^n-5} = 2\cdot 10^{10^n-6+n},\]
implying that $f_d(x) = x$, which concludes the proof.
\end{proof}

Below is the smallest number $x$ for which the number of \boxed{10}s as substrings of the numbers in the range $\inclrange{1}{x}$ is more than or equal to $x$. And by a lucky strike, the equality holds. The number has 93 digits and doesn't fit on a line. Fortunately, the middle part of the number consists of a long run of nines, namely 88 of them. So we replaced some of the nines with dots without losing information. The number is:
\[a_\ge(\boxed{10}) = \numprint{109999999999999999999999}\cdots \numprint{999999999999999999999810}.\]
Now the reader can do an exercise and find the corresponding number for the ``more than'' sequence, $a_>(\boxed{10})$.

The value of $a_{\ge}(\boxed{11})$ miraculously has 93 digits with a middle run of 88 nines:
\[a_\ge(\boxed{11}) = \numprint{119999999999999999999999}\cdots \numprint{999999999999999999999811}.\] 
Note how strikingly similar it is to the tenth element of the sequence! Can you explain that similarity between $a_{\ge}(\boxed{10})$ and $a_{\ge}(\boxed{11})$?

Sadly, $a_{\ge}(\boxed{12})$ is not so pretty, though it still has a middle run of 68 nines allowing us to display it on the line using dots. The total number of digits is 94:
\[a_{\ge}(\boxed{12}) = \numprint{1296624070230872986615199999999}\cdots \numprint{999999999999812}.\] 
It appears that we are lucky again, that the $a_=$ sequence is the same as the $a_\ge$ sequence for \boxed{11} and \boxed{12}. Our luck runs out at \boxed{21}.

We have calculated the values up to $d$=\boxed{113} so far, but please feel free to contribute more compute! 

The values for $a_=(50)$ and $a_=(99)$ are the same nice, round number: $10^{101}$. Similarly, $a_=(999) = 10^{1002}$. It turns out that and $a_>(99) = a_=(99)-1$ and $a_>(999) = a_=(999)-1$, but there is no such pretty relationship for \boxed{50} or \boxed{500}.  There are, however, 6 two-digit values for which the difference is 100. There are also 8 two-digit values that share an $a_=$ value of $9465 \cdot 10^{97}$, reminiscent of the $a_=$ values for 7 and 8: $9465 \cdot 10^6$. And would you have guessed, there are corresponding three-digit values whose $a_=$ values are $9465 \cdot 10^{998}$! There are other tantalizing patterns; for details please see the code~\cite{MartonKhovanova2023}, and share what you find!

One would expect three-digit stickers to occur ten times less frequently than two-digit stickers, and indeed the corresponding values in this sequence that we've computed are in the neighborhood of a thousand digits. Each answer for a three-digit sticker has taken about an hour of compute to find, compared to a few seconds for two-digit numbers, so we haven't checked them all. 

Now that we defined our sequence for any sticker $d$, we get some natural sequences. The first one is the sequence of stickers $d$ for which $a_>(d) = a_\ge(d)$.:
\[5,\ 6,\ 7,\ 8,\ 9,\ 21,\ 24,\ 29,\ 33,\ 39,\ 50,\ 52,\ 55,\ 56,\ 58,\ 59,\ 63,\ 66,\ 67,\ \dots. \] Another lens on this sequence is those d for which $a_=(d)$ exists and $a_=(d) < a_>(d)$, which is almost the complement, except that 0 still does not appear: \[1, 2, 3, 4, 10, 11, \dots. \]

We can also extend Table~\ref{table:E_d}, with the value for \boxed{10} being \numprint{5352172560} followed by 90 zeroes, but the value for \boxed{11} is again not so pretty: values up to \boxed{94} are given in the supplementary materials~\cite{MartonKhovanova2023}.

Finally, we extend \A{130432} to the multi-digit case. Starting with the value for the \boxed{10} sticker, these are:
\[3167,\ 9043,\ 7485,\ 1305,\ 5299,\ 297,\ 4659,\ 1019,\ 37,\ 2019,\ 617,\ 621,\ \ldots.\]
Sadly, this sequence has lost the divisibility property: \A{130432}$(d) + 1$ is guaranteed to be divisible by $d$ only for $d < 10$.

\section{All Your Base}

Of course the sticker sheets that came with VHS tapes had letters too. Interestingly, some sticker sheets (e.g., Figure~\ref{fig:vhs-image}) had letters A through F, that seemed to beg for hexadecimal numbering, though other sheets included the full alphabet. The algorithms generalize straightforwardly to any base, substituting base $b$ where we previously wrote $10$. 
\begin{figure}[ht!]
    \centering    \includegraphics[width=0.5\linewidth]{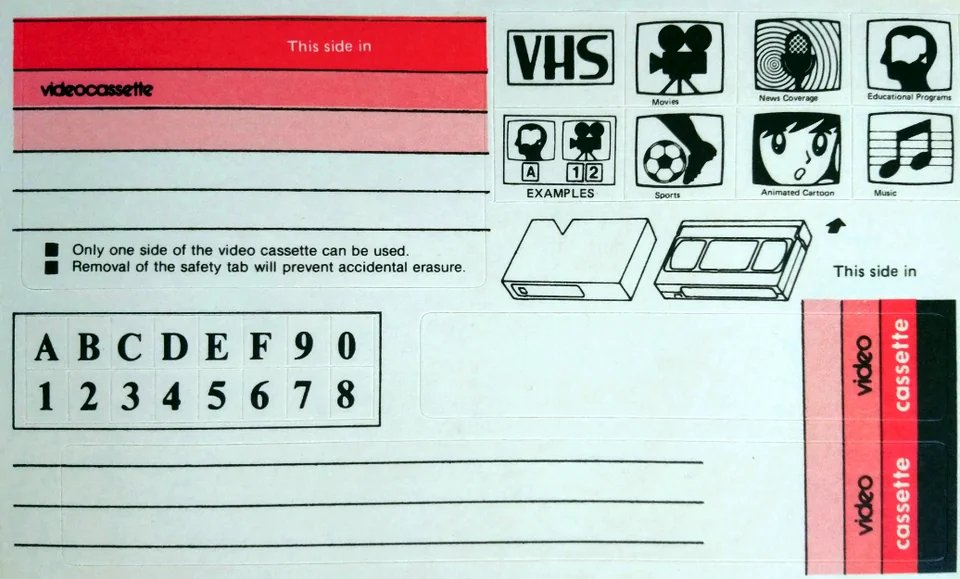}
    \caption{One of the sticker sheets that came with early VHS tapes, with gratitude to an r/nostalgia user~\cite{vhs-image}.}
    \label{fig:vhs-image}
\end{figure}

Let us add base $b$ as the second parameter of our functions. For example, we denote by $f_d(x,b)$ the number of times the sticker $d$ is used in the writing of numbers in the range $\inclrange{1}{x}$ in base $b$. Similarly, we add the base to functions $a$: $a_>(d, b)$, $a_=(d, b)$, and $a_\ge(d, b)$, where we assume that sticker $d$ is also written in base $b$.

The unary base, where $b=1$, is a special case, as only stickers containing ones are relevant, and the function $f_d(x,1)$ can be calculated explicitly. For example, $f_1(x,1) = \frac{x(x+1)}{2}$. We leave it for the reader to investigate multi-digit cases in this base, and from now on, we will assume that $b > 1$.

Two sequences related to different bases are already in the database.
\begin{itemize}
\item Sequence \A{092175}$(b)$ represents our sequence $a_>(1,b)$. Starting from base 1, the sequence $a_>(1, b)$ 
progresses as follows:
\[2,\ 3,\ 13,\ 29,\ 182,\ 427,\ \numprint{3931},\ \numprint{8185},\ \numprint{102781},\ \numprint{199991},\ \numprint{3179143},\ \ldots.\]

Comfortingly, \A{092175}$(10) = \numprint{199991}$, which we already knew from Table~\ref{table:2seq}.
\item Sequence \A{165617}$(b)$ counts the number of solutions to $f_1(x, b) = x$. Sequence \A{165617} starts from $b=2$ as
\[2,\ 4,\ 8,\ 4,\ 21,\ 5,\ 45,\ 49,\ 83,\ 10,\ 269,\ 11,\ 202,\ 412,\ 479,\ 15,\ \ldots,\]
and, not surprisingly, the ninth term is 83, which we already knew from Table~\ref{table:seqnum}.
\end{itemize}

By the way, sequence \A{165617} is easy to calculate, because the largest possible number such that $f_1(x, b) = x$ is known. This number is the concatenation of $b-1$ ones followed by a single zero written in base $b$, see the comment in sequence \A{165617}. Expressed in base 10, these largest numbers are (starting from index 2):
\[2,\ 12,\ 84,\ 780,\ \numprint{9330},\ \numprint{137256},\ \numprint{2396744},\ \numprint{48427560},\ \numprint{1111111110},\ \ldots,\]
and they are in the database as sequence \A{226238}.

We promised to show that the solution to $f_d(x) = x$ for a one-digit nonzero sticker \boxed{d} in base 10 doesn't exceed $d\cdot 10^{10}$. We waited for this moment to do the proof for any base $b$.

\begin{proposition}
\label{prop:bounds}
For any digit $d > 0$ in base $b > d$ the maximum possible value of $a_=(d, b)$ is $b^b$ and all $x$ such that $f_d(x, b)=x$ must be $\leq d \cdot b^b$.
\end{proposition}

\begin{proof}
Similar to base 10, we can calculate that $f_b(b^b) = b^b$, proving that $a_=(d, b) \le b^b$. If $x = d \cdot b^b$, then $f_d(x,b)=x + 1$. All numbers in the range $\inclrange{d \cdot b^b}{(d+1) b^b}$ have $d$ for the first digit, implying that there are no solutions to $f_d(x,b) = x$ in this range. Then $f_d((d+1) b^b) = (d+2) b^b$. Converting Lemma~\ref{lemma:step} to base $b$, we see that no solution can appear among the next $b^b$ numbers, while the next $b^b$ numbers use at least $b^b$ digits $d$. By repeating this ad infinitum, the conclusion follows.
\end{proof}

We can generalize Theorem~\ref{thm:welldefined} to any base $b > 2$.

\begin{theorem}
The value $a_=(d,b)$ is well-defined for any $b > 2$ and any $d > 0$. For $b = 2$, it is well-defined when $d> 0$ is not a power of 2.
\end{theorem}

\begin{proof}
    By changing 10 to base $b$ in all the right places, the Theorem~\ref{thm:welldefined} can be adjusted for any $b$ and $d$, except when $b = 2$, and $d$ is a power of $b$.
\end{proof}

As you might have noticed, the theorem above excludes cases when $b = 2$ and $d$ is a power of 2. We ran our program for those cases, with findings shown in Table~\ref{table:base2_powerof_2}, and confirmed that indeed not all seem to exist. We have not proven an upper bound, but checked for 7 up to 1200 decimal digits, far larger than the solutions found for larger powers.

\begin{table}[ht!]
\begin{center}
\begin{tabular}{r r r r}
\multicolumn{2}{c}{$d$} & $\{x : f_d(x, 2) = x\}$ & bits in $a_=(d, 2)$\\
2 & \boxed{10} & 21 & 5\\
4 & \boxed{100} & 610 & 10\\
8 & \boxed{\numprint{1000}} & \numprint{283187} & 19\\
16 & \boxed{\numprint{10000}} & \numprint{35609822115} & 36\\
32 & \boxed{\numprint{100000}} & \numprint{300185978028231432373} & 69\\
64 & \boxed{\numprint{1000000}} & \textrm{unique value} & 134 \\
128 & \boxed{\numprint{10000000}} & \multicolumn{2}{c}{\hspace*{8em}\textrm{not found!}} \\
256 & \boxed{\numprint{100000000}} & \textrm{unique value} & 520 \\
512 & \boxed{\numprint{1000000000}} & \textrm{unique value} & 1033 \\
1024 & \boxed{\numprint{10000000000}} & \textrm{1023 consecutive values} & 2058 \\

\end{tabular}
\caption{Solutions for $f_d(x, 2) = x$ for stickers $\boxed{d}$ being the binary stickers corresponding to powers of 2. Larger values are too long to show here.}
\label{table:base2_powerof_2}
\end{center}
\end{table}

We now turn our attention to $d=0$. Many of the values $a_=(0, b)$ are undefined. To be sure, we need to check to some upper bound, and it need not be tight, but how far do we need to check?

\begin{proposition}
\label{prop:zerobounds}
For digit 0 in base $b > 1$, the value of $a_=(0, b)$, if it is well-defined, must be less than $b^{b+3}$.
\end{proposition}

\begin{proof}
Similar to base 10, it is enough to find a number $y > b^b$, such that $f_0(y,b) > y + b^b$. We are then guaranteed that there are no solutions to $f_0(x,b) = x$, for $x > y$.

The number of zeros used in range $\inclrange{b^{p-1}}{b^{p} - 1}$ is $(p-1)(b-1)b^{p-2}$. When $p = b+3$, then the range contains $(b+2)(b-1)b^{b+1} = b^{b+3} + (b-2)b^{b+1} \ge b^{b+3} + b^b$ zeros, whenever $b > 2$. 

If $b=2$, then $b^{b+3} = 32$. The number of zeros used in the range $\inclrange{1}{32}$ is 54, which is greater than $32+4$.

Thus, for $b>1$, and $y = b^{b+3}$, we have $f_0(y,b) > y + b^b$, implying that $a_=(0,b) < y$.
\end{proof}

Those bases in which $a_=(0, b)$ does not exist are now \A{364972}: 
\[3, 4, 5, 6, 7, 8, 9, 10, 12, 14, 15, 17, 18, 19, 20, \ldots\]
The first few values of $a_=(0, b)$ where it does exist are shown in Table~\ref{table:zero_in_bases}.

\begin{table}[ht!]
\begin{center}
\begin{tabular}{rr}
base $b$ & $a_=(0, b)$ \\
2& 8 \\
11& \numprint{3152738985031} \\
13& \numprint{3950024143546664} \\
16& \numprint{295764262988176583799} \\
24& \numprint{32038681563209056709427351442469835} \\
26& \numprint{160182333966853031081693091544779177187} \\
28& \numprint{928688890453756699447122559347771300777482} \\
29& \numprint{74508769042363852559476397161338769391145562} \\
31& \numprint{529428987529739460369842168744635422842585510266}
\end{tabular}
\caption{Values of the ``exactly'' sequence for the zero sticker in the first few bases where it exists.}
\label{table:zero_in_bases}
\end{center}
\end{table}

\section{Future research}

We have stretched the notion of a sticker with the multi-digit case, because although our definition has nice mathematical properties it uses stickers that are counted as overlapping, and excludes multi-digit stickers that begin with a zero. These are of course both decisions that one can relax to explore fresh possibilities.

The common inclusion of letter stickers on these VHS sticker sheets invites not only generalization to other bases, but also generalization to, dare we say it, textual labels? But this is still a math paper, so let's not get carried away: we can restrict ourselves by imagining textual labels that correspond to spelled-out versions of the numbers! Thus, for some English-like convention $c$,

\begin{tabular}{r r l}
$a_=(``o", c)$&$=2$&``One", ``twO"),\\
$a_=(``e", c)$&$=3$&(``onE", ``two", ``thrEE") \\
$a_>(``j", c)$&undefined& ``J'' does not appear.\\
\end{tabular}

Of course there are a jillion variations by language, locale, etc.

Throughout the paper we've noted phenomena that seemed interesting but bear further investigation. Examples include the many solutions to $a_=$ in base 10 that begin with ``9465", what governs whether $a_=$ is smaller or greater than $a_>$, what happens in unary base, whether the sequences are well-defined for all powers of 2 in base 2, and whether all of those have unique solutions. We encourage you to look at the tables in the supplementary materials, as there are many more patterns begging for attention.

One such pattern we see in the data, but have not proven anything about is that the number of digits in $a_=(b, b)$ equals $b^2 + b + 3$, for $b>2$. The solutions, expressed in their respective bases in Table~\ref{table:10inBases}, all have a similar form for $b>2$: the first digits are 1, then 0, then digits $b-1$, finishing with $b-2$, 1, 0.

\begin{table}[ht!]
\begin{tabular}{r r r}
base & length & $a_=(b, b)$ or equivalently, $a_=(\boxed{10}, b)$ \\
2 & 5 & $\numprint{10101}_2$, \\
3 & 9 & $\numprint{102222110}_3$,\\
4 & 15 & $\numprint{103333333333210}_4$,\\
5 & 23 & $\numprint{10444444444444444444310}_5$,\\ 
6 & 33 & $\numprint{105555555555555555555555555555410}_6$. 
\end{tabular}
\caption{Lengths and values of $a_=(\boxed{10}, b)$ for the first few bases $b$.}
\label{table:10inBases}
\end{table}

And of course the sequences we described in this paper can be extended, and there are many related sequences to be cataloged. We would love to hear tales from your explorations. Enjoy the sequence hunt!

\section{Acknowledgments}

We are grateful to Alexey Radul for his helpful suggestions. We are also thankful to the anonymous reviewers of the American Mathematical Monthly for encouraging us to dig deeper into the topic and providing helpful ideas and suggestions.

\end{document}